\documentclass[reqno,11pt]{amsart}
\usepackage{amsmath,amsthm,amscd,amssymb,graphicx,enumerate,latexsym}

\usepackage[raiselinks,colorlinks]{hyperref}
\hypersetup{citecolor=blue}

\numberwithin{equation}{section}

\theoremstyle{plain}
\newtheorem{thm}{Theorem}[section]
\newtheorem{lemma}[thm]{Lemma}
\newtheorem{cor}[thm]{Corollary}
\newtheorem{conj}[thm]{Conjecture}

\theoremstyle{definition}

\theoremstyle{remark}

\def\Re{\mathop{\rm Re}\nolimits}

\newcommand{\s}{\text{\rm{s}}}

\allowdisplaybreaks

\title{On a Conjecture for Higher-order Szeg\H o theorems}
\author{Milivoje Lukic}
\address{Rice University, 6100 Main Street, Mathematics MS 136, Houston, TX 77005}
\date\today
\email{milivoje.lukic@rice.edu}


\begin{document}

\begin{abstract}
We disprove a conjecture of Simon for higher-order Szeg\H o theorems for orthogonal polynomials on the unit circle and propose a modified version of the conjecture.
\end{abstract}

\maketitle
\section{Introduction}

In this paper we investigate probability measures $\mu$ supported on the unit circle $\partial\mathbb{D} = \{ e^{i\theta}\mid \theta \in [0,2\pi)\}$. If $\mu$ has infinite support, the sequence $1,z,z^2,\dots$ is linearly independent in $L^2(\partial\mathbb{D},d\mu)$, so Gram--Schmidt orthogonalization provides orthonormal polynomials $\varphi_n(z)$, which obey the recursion relation
\[
z\varphi_n(z)=\sqrt{1-\lvert \alpha_n \rvert^2} \varphi_{n+1}(z) + \bar \alpha_n \varphi_n^*(z)
\]
with $\varphi^*_n(z)=z^n \overline{\varphi_n(1/\bar z)}$ and  coefficients $\alpha_n \in \mathbb{D}$ called Verblunsky coefficients. Thus, to the measure $\mu$ there corresponds the sequence $\alpha = \{\alpha_n\}_{n=0}^\infty \in \mathbb{D}^\infty$. This is, by Verblunsky's theorem \cite{Verblunsky35}, a bijective correspondence; see \cite{OPUC1,OPUC2} for more information.

With respect to Lebesgue measure on $\partial\mathbb{D}$, $\mu$ can be decomposed into an absolutely continuous and a singular part,
\[
d\mu = w(\theta) \frac{d\theta}{2\pi} + d\mu_\s.
\]
The celebrated Szeg\H o theorem for the unit circle (due in this generality to Verblunsky~\cite{Verblunsky36})
 states that $\alpha \in \ell^2$ is equivalent to 
\begin{equation}\label{1.1}
\int \log w(\theta) \frac{d\theta}{2\pi} > - \infty.
\end{equation}
Note that,  since $\log w \le w -1$, \eqref{1.1} is equivalent to $\log w \in L^1(\partial \mathbb{D}, \frac{d\theta}{2\pi})$. This theorem has led the way for many related results on orthogonal polynomials and Schr\"odinger operators; see \cite{Rice} for a book-length treatment. In this paper we focus on higher-order Szeg\H o theorems, where \eqref{1.1} is replaced by a weaker condition. A conjecture from \cite{OPUC1} describes the situation with finitely many singularities. Denote by $S$ the shift operator on sequences,
\[
(S x)_n = x_{n+1}.
\]

\begin{conj}[{\cite[Section~2.8]{OPUC1}}]\label{C1.1}
Let $m_1, \dots, m_l \in \mathbb{N}$, and let $\theta_1, \dots, \theta_l$ be distinct elements of $[0,2\pi)$. Then
\begin{equation} \label{1.2}
\int \prod_{k=1}^l (1-\cos(\theta-\theta_k))^{m_k} \log w(\theta) \frac{d\theta}{2\pi} > -\infty 
\end{equation}
is equivalent to
\begin{equation}\label{1.3}
\prod_{k=1}^l (S-e^{-i\theta_k})^{m_k} \alpha \in \ell^2 \quad \text{and}\quad \alpha \in \ell^{2\max_k (m_k)+2}.
\end{equation}
\end{conj}

The conjecture originated with Simon's proof \cite[Section~2.8]{OPUC1} for the case $\sum_{k=1}^l m_k = 1$; Simon--Zlato\v s \cite{SimonZlatos05} proved the case $\sum_{k=1}^l m_k = 2$, and Golinskii--Zlato\v s \cite{GolinskiiZlatos07} proved the equivalence under the assumption $\alpha \in \ell^4$. However, we will see that it is not true in general.

In Theorem~\ref{T2.1}, we show that \eqref{1.3} is equivalent to the existence of sequences $\beta^{(1)}, \dots, \beta^{(l)}$ such that
\begin{equation}\label{1.4}
\alpha = \beta^{(1)} +\dots + \beta^{(l)}
\end{equation}
and, for all $k$,
\begin{align}
(S-e^{-i\theta_k})^{m_k} \beta^{(k)} & \in \ell^2 \label{1.5} \\
\beta^{(k)} & \in \ell^{2\max_j (m_j)+2} \label{1.6}
\end{align}

We find this decomposition natural because for each critical point $e^{i\theta_k}$ in \eqref{1.2}, there is a sequence $\beta^{(k)}$ which corresponds to this critical point through \eqref{1.5}. However, the conjectured equivalence of \eqref{1.2} and \eqref{1.4}--\eqref{1.6} implies that the degree $m_k$ associated with $e^{i\theta_k}$ affects not only the conditions on the corresponding $\beta^{(k)}$, but all on the $\beta^{(j)}$ through the $\max$ in \eqref{1.6}. We find it more natural to replace \eqref{1.6} by
\begin{equation}\label{1.7}
\beta^{(k)} \in \ell^{2m_k + 2}.
\end{equation}
This suggests a modified version of the conjecture.

\begin{conj}\label{C1.2}
Let $m_1, \dots, m_l \in \mathbb{N}$, and let $\theta_1, \dots, \theta_l$ be distinct elements of $[0,2\pi)$. Then \eqref{1.2} is equivalent to the existence of sequences $\beta^{(1)}, \dots, \beta^{(l)}$ such that \eqref{1.4}, \eqref{1.5}, \eqref{1.7} hold.
\end{conj}

Although this conjecture is incompatible with Conjecture~\ref{C1.1}, they coincide in all the cases which have been proved. However, in the very next case one would naturally proceed to verify, the conjectures differ. For the condition
\begin{equation}\label{1.8}
\int  (1-\cos\theta)^2 (1+\cos\theta) \log w(\theta) \frac{d\theta}{2\pi} > -\infty,
\end{equation}
Conjecture~\ref{C1.1} predicts necessary and sufficient conditions
\[
(S-1)^2(S+1)\alpha \in \ell^2, \quad \alpha\in \ell^6.
\]
Conjecture~\ref{C1.2} is stated in terms of $\beta$'s, but as we will explain in Section~\ref{S2}, it can be restated in terms of $\alpha$'s and predicts that \eqref{1.8} is equivalent to
\begin{equation} \label{1.9}
(S-1)^2(S+1) \alpha \in \ell^2, \quad \alpha \in \ell^6, \quad (S-1)^2\alpha\in \ell^4.
\end{equation}
Due to the prohibitive nature of the calculations involved, we only prove a special case.

\begin{thm}\label{T1.3}
Let $(S-1)(S+1)\alpha\in\ell^2$ and $\alpha \in \ell^6$. Then
\eqref{1.8} is equivalent to $ (S-1)^2 \alpha \in \ell^4$.
\end{thm}

In particular, Theorem~\ref{T1.3} disproves Conjecture~\ref{C1.1}:

\begin{cor}
Let the Verblunsky coefficients of the measure $\mu$ be given by
\begin{equation}
\alpha_n = \frac {1+(-1)^n}{3 (n+1)^{1/4}}.
\end{equation}
Then $(S-1)^2(S+1)\alpha \in \ell^2$ and $\alpha \in \ell^6$, but
\begin{equation}\label{1.11}
\int  (1-\cos\theta)^2 (1+\cos\theta) \log w(\theta) \frac{d\theta}{2\pi} = -\infty.
\end{equation}
\end{cor}

\begin{proof}
It is straightforward to verify $\alpha\in \ell^6$ and $(S-1)(S+1)\alpha \in \ell^2$; the latter also implies $(S-1)^2(S+1)\alpha \in \ell^2$. However, $(S-1)^2\alpha \notin \ell^4$, so Theorem~\ref{T1.3} implies \eqref{1.11}.
\end{proof}

Necessary and sufficient conditions for \eqref{1.2} in terms of Verblunsky coefficients have been proved by Denisov--Kupin~\cite{DenisovKupin06}, following work of Nazarov--Peher\-storfer--Volberg--Yuditskii~\cite{NazarovPeherstorferVolbergYuditskii05} for Jacobi matrices. However, these conditions are in a more complicated form which hasn't been successfully related to conditions such as those discussed here. For Jacobi matrices, analogs of the cases $\sum_{k=1}^l m_k = 1,2$ have been proved by Laptev--Naboko--Safronov~\cite{LaptevNabokoSafronov03} and Kupin~\cite{Kupin04}.

We end on a pessimistic note: even though all the existing results, including Theorem~\ref{T1.3}, are compatible with Conjecture~\ref{C1.2}, we are not confident that it is true, either. In \cite{Lukic1}, we analyzed Verblunsky coefficients of the form \eqref{1.4}, with \eqref{1.5} replaced by the stronger condition $(S-e^{-i\theta_k})\beta^{(k)} \in \ell^1$, and with $\alpha \in \ell^p$ for some $p<\infty$. There, the measure is purely absolutely continuous except on an explicit finite set of points. However, this set of possible pure points increases with increasing $p$, and can contain points not in $\{e^{i\theta_k} \mid k=1,\dots,K\}$ if $p>3$. The possibility of these points was shown by Kr\"uger \cite{Kruger12} and Lukic \cite{Lukic5}. If the analogous phenomenon is true here, it would mean that for large enough $m_k$, the measure may have points outside of $\{e^{i\theta_k} \mid k=1,\dots,K\}$ where $\log w$ is not locally $L^1$, so \eqref{1.2} would be false. 

\section{Decomposition}\label{S2}

If $P_1, \dots, P_l \in \mathbb{C}[x]$ are pairwise coprime, then there exist polynomials $U_1, \dots, U_l \in \mathbb{C}[x]$ such that
\begin{equation}\label{2.1}
\sum_{j=1}^l U_j \prod_{i\neq j} P_i = 1.
\end{equation}
This is easily proved by induction on $l$, since $\mathbb{C}[x]$ is a principal ideal domain. 

\begin{thm}\label{T2.1}
Fix $2\le  p<\infty$. Let $P_1,\dots, P_l \in \mathbb{C}[x]$ be pairwise coprime and $U_1,\dots, U_l \in \mathbb{C}[x]$ be such that \eqref{2.1} holds. Then the following are equivalent:
\begin{enumerate}[(i)]
\item $\alpha \in \ell^p$ and $P_1(S) \cdots P_l(S) \alpha \in \ell^2$;
\item if we define $\beta^{(j)} = U_j(S) \prod_{i\neq j} P_i(S) \alpha$ for $j=1,\dots,l$, then $\beta^{(j)} \in \ell^p$, $P_j(S) \beta^{(j)} \in \ell^2$ and
\begin{equation}\label{2.2}
\alpha = \beta^{(1)}+ \dots + \beta^{(l)}
\end{equation}
\item there exist sequences $\beta^{(1)}, \dots, \beta^{(l)} \in \ell^p$ such that \eqref{2.2} holds and that $P_j(S) \beta^{(j)} \in \ell^2$ for $j=1,\dots,l$.
\end{enumerate}
\end{thm}

\begin{proof} This proof repeatedly uses the following obvious fact: if $\gamma \in \ell^n$, then $Q(S) \gamma \in \ell^n$, for an arbitrary polynomial $Q\in \mathbb{C}[x]$.

(i) implies (ii): $P_j(S) \beta^{(j)} = U_j(S) P_1(S) \cdots P_l(S) \alpha \in \ell^2$, and $\alpha \in \ell^p$ implies $\beta^{(j)} \in \ell^p$. Finally, \eqref{2.1} implies \eqref{2.2}.

(ii) implies (iii) trivially.

(iii) implies (i): since $\beta^{(1)}, \dots,\beta^{(l)}$ are in $\ell^p$, so is their sum $\alpha$. Further, $P_j(S) \beta^{(j)} \in \ell^2$ implies $P_1(S) \cdots P_l(S)  \beta^{(j)} \in \ell^2$, and  summing in $j$, $P_1(S) \cdots P_l(S)  \alpha \in \ell^2$.
\end{proof}

For example, in the case considered by Theorem~\ref{T1.3}, take $P_1(z) = (z-1)^2$, $P_2(z) = z+1$. Note that
\[
\tfrac 14 (z-1)^2 - \tfrac 14 (z-3)(z+1) =1,
\]
so take $U_1(z)=-\tfrac 14(z-3)$, $U_2(z) = \tfrac 14$. Then
\[
\beta^{(1)} = - \frac 14 (S-3)(S+1) \alpha, \qquad \beta^{(2)} = \frac 14 (S-1)^2 \alpha.
\]
Notice that $(S-1)^2(S+1) \alpha \in \ell^2$ is equivalent to $(S-1)^2 \beta^{(1)}, (S+1)\beta^{(2)} \in \ell^2$. Further, $\beta^{(1)} \in \ell^6$ and $\beta^{(2)} \in \ell^4$ is equivalent to $\alpha \in \ell^6$ and $\beta^{(2)} \in \ell^4$. Thus, Conjecture~\ref{C1.2} predicts that \eqref{1.8} is equivalent to \eqref{1.9}.

\section{Proof of Theorem~\ref{T1.3}}\label{S3}

Our proof of Theorem~\ref{T1.3} follows the method of Simon~\cite[Section~2.8]{OPUC1}, Simon--Zlato\v s~\cite{SimonZlatos05} and Golinskii--Zlato\v s~\cite{GolinskiiZlatos07}. Define
\begin{equation}\label{3.1}
Z(\mu) = \int (1-\cos \theta)^2(1+\cos\theta) \log w(\theta) \frac{d\theta}{2\pi}.
\end{equation}
Let us assume for a moment that $\alpha \in \ell^2$. Then $\log w \in L^1$ by Szeg\H o's theorem; denote by $w_m$ the moments of $\log w(\theta)$,
\[
w_m = \int e^{-im\theta} \log w(\theta) \frac{d\theta}{2\pi},
\]
noting that $w_{-m} = \bar w_m$. The first few moments are computed in \cite{GolinskiiZlatos07},
\begin{align*}
w_0 & =  \sum_k \log \rho_k^2 \\
w_1 & =  - \sum_k \alpha_k \bar\alpha_{k-1} \\
w_2 & = \sum_k \bigl( -\alpha_k \bar\alpha_{k-2}\rho_{k-1}^2 + \tfrac 12 \alpha_k^2 \bar \alpha_{k-1}^2 \bigr) \\
w_3 & = \sum_k \bigl( -\alpha_k \bar\alpha_{k-3}\rho_{k-1}^2 \rho_{k-2}^2 + \alpha_k^2 \bar\alpha_{k-1} \bar\alpha_{k-2} \rho_{k-1}^2 + \alpha_k \alpha_{k-1} \bar\alpha_{k-2}^2 \rho_{k-1}^2 - \tfrac 13 \alpha_k^3 \bar \alpha_{k-1}^3 \bigr)
\end{align*}
This uses the convention $\alpha_{-1}=-1$ and $\alpha_k = 0$ for $k\le -2$, and $\rho_k = \sqrt{1-\lvert \alpha_k\rvert^2}$. Since
\[
(1-\cos \theta)^2(1+\cos\theta) = \tfrac 18 (4 - e^{i\theta} - e^{-i\theta} - 2 e^{2i\theta} - 2 e^{-2i\theta} + e^{3i\theta} + e^{-3i\theta}),
\]
\eqref{3.1} implies 
\begin{equation*}
Z(\mu) =  \tfrac 14 \Re ( 2 w_0 - w_1 - 2 w_2  + w_3  ).
\end{equation*}

Thus,  for measures with $\alpha\in \ell^2$,
\begin{align}
Z(\mu) & = \tfrac 14 \sum_k  \Re \Bigl( 2 \log \rho_k^2 + \alpha_k  \bar \alpha_{k-1} + 2 \alpha_{k} \bar \alpha_{k-2} \rho_{k-1}^2 - \alpha_{k}^2 \bar \alpha_{k-1}^2 - \alpha_{k} \bar \alpha_{k-3} \rho_{k-1}^2 \rho_{k-2}^2 \nonumber \\ 
& \qquad \qquad + \alpha_{k}^2 \bar\alpha_{k-1} \bar \alpha_{k-2} \rho_{k-1}^2 + \alpha_{k}\alpha_{k-1} \bar \alpha_{k-2}^2 \rho_{k-1}^2 - \tfrac 13  \alpha_{k}^3 \bar \alpha_{k-1}^3 \Bigr) \label{3.2}
\end{align}

Now let $\mu$ be arbitrary. Let $\mu_n$ be the measure with Verblunsky coefficients
\begin{equation*}
\alpha^{(n)} = (\alpha_0,\alpha_1,\dots,\alpha_{n-1},0,0, \dots).
\end{equation*}
The $\mu_n$ are known as Bernstein--Szeg\H o approximations of $\mu$, and $\mu_n$ converge weakly to $\mu$. It is proved in \cite{GolinskiiZlatos07} that 
\begin{equation*}
Z(\mu) = \lim_{n\to\infty} Z(\mu_n),
\end{equation*}
or equivalently, that the formula \eqref{3.2} holds for $\mu$ as well (since $Z(\mu_n)$ is just the partial sum for the right-hand side of \eqref{3.2}). 

To rewrite \eqref{3.2} in a more useful way, take
\begin{align*}
L_k & = 2\log(1-\lvert\alpha_k\rvert^2) + 2 \lvert\alpha_k\rvert^2 + \lvert\alpha_k\rvert^4 \\
E_k & =  - \tfrac 12 (1-\lvert \alpha_{k-1}\rvert^2 -\lvert \alpha_{k-2}\rvert^2) \lvert \alpha_{k} -\alpha_{k-1} - \alpha_{k-2} + \alpha_{k-3} \rvert^2 \\
G_k &  = - \tfrac 1{32} \lvert \alpha_{k} - 2 \alpha_{k-1} + \alpha_{k-2} \rvert^4 \\
J_k & =  \Bigl(  \tfrac 34 \alpha_k \bar\alpha_{k-1}\bar\alpha_{k-2} + \tfrac 54 \lvert \alpha_{k-2}\rvert^2 \bar\alpha_{k-1}  + \tfrac 98 \lvert\alpha_k\rvert^2 \bar\alpha_{k-2} + \tfrac 12 \bar\alpha_{k-1}^2   \alpha_{k-2}  + \bar\alpha_{k-2} \lvert\alpha_{k-1}\rvert^2  \\
& \qquad  + \tfrac {23}{16} \lvert \alpha_{k-2}\rvert^2 \bar \alpha_{k-2}  - \tfrac 54 \lvert\alpha_k\rvert^2  \bar\alpha_{k-1}   - \tfrac 34 \bar\alpha_k \bar \alpha_{k-1} \alpha_{k-2}  + \tfrac 1{16} \alpha_k \bar \alpha_{k-2}^2    + \tfrac 14 \lvert \alpha_{k-2}\rvert^2 \bar\alpha_k  \Bigr) \\
H_k &  = (\alpha_{k} - \alpha_{k-2}) J_k \\
F_k & =  - \alpha_k \bar\alpha_{k-3} \lvert \alpha_{k-1}\rvert^2 \lvert \alpha_{k-2}\rvert^2 - \alpha_k^2 \bar\alpha_{k-1} \bar\alpha_{k-2} \lvert \alpha_{k-1}\rvert^2  - \alpha_{k} \alpha_{k-1} \bar\alpha_{k-2}^2 \lvert\alpha_{k-1}\rvert^2 - \tfrac 13 \alpha_k^3 \bar\alpha_{k-1}^3 \\
I_k & = \Bigl( - \tfrac 32 \lvert \alpha_k\rvert^2 - \lvert \alpha_{k-1}\rvert^2 - \tfrac 12 \lvert \alpha_{k-2}\rvert^2 + \alpha_k \bar\alpha_{k-2} + \alpha_{k-1} \bar\alpha_{k-2} + \tfrac 12 \lvert\alpha_k\rvert^2 \lvert\alpha_{k-2}\rvert^2  \\
& \qquad - \tfrac{31}{32} \lvert\alpha_k\rvert^4 -\tfrac{31}{32} \lvert\alpha_{k-1}\rvert^4 - \tfrac 34 \alpha_k^2 \bar\alpha_{k-1}^2+ \lvert\alpha_k\rvert^2 \alpha_k \bar\alpha_{k-1} - \lvert\alpha_{k-1}\rvert^2 \alpha_{k-1} \bar\alpha_{k-2}\\
& \qquad  - \lvert\alpha_{k-1}\rvert^2 \alpha_k \bar\alpha_{k-2} - \lvert\alpha_k\rvert^2 \alpha_k \bar\alpha_{k-2} - \lvert\alpha_k\rvert^2 \alpha_{k-1}\bar\alpha_{k-2} + \tfrac 12 \lvert\alpha_{k-1}\rvert^2 \lvert\alpha_{k-2}\rvert^2 \Bigr)
\end{align*}

\begin{lemma}\label{L3.1} Let $\alpha \in \ell^6$ and $(S^2-1)\alpha \in \ell^2$. Then $\{L_k\}, \{E_k\}, \{H_k\}, \{F_k\} \in \ell^1$.
\end{lemma}

\begin{proof}
$(S^2-1)\alpha\in \ell^2$ implies $(S^3-S^2-S+1)\alpha = (S-1)(S^2-1)\alpha \in \ell^2$. Thus, $\{E_k\}\in \ell^1$, since $\lvert\alpha_k \rvert < 1$ for all $k$.

$\alpha \in \ell^6$ implies $\{J_k\} \in \ell^2$, so together with $(S^2-1)\alpha \in \ell^2$ it implies $\{H_k\}\in\ell^1$. $\alpha \in \ell^6$ also implies $\{F_k\} \in \ell^1$.

$\alpha \in \ell^6$ implies that $\lvert\alpha_k\rvert < \tfrac 12$ for all but finitely many $k$. For $z\in [0,\tfrac 14]$, we have the uniform estimate
\[
\left\lvert \log(1-z) + z + \tfrac 12 z^2 \right\rvert \le C z^3
\]
for some finite $C$. Take $z=\lvert \alpha_k\rvert^2$  to conclude that $\lvert L_k \rvert \le 2 C \lvert \alpha_k\rvert^6$ for all but finitely many $k$; thus, $\alpha \in \ell^6$ implies $\{L_k\} \in \ell^1$.
\end{proof}

A straightforward calculation shows that
\begin{equation}\label{3.3}
Z(\mu)  = \tfrac 14 \sum_k \Re \Bigl(  L_k + E_k + G_k + H_k  + F_k + I_k - I_{k-1} \Bigr).
\end{equation}
However, $\sum_k (I_k-I_{k-1})=0$, since it is a telescoping sum and $\lim_{k\to\pm \infty} I_k = 0$.
By \eqref{3.3} and Lemma~\ref{L3.1}, if $\alpha\in \ell^6$ and $(S^2-1)\alpha \in \ell^2$, then
\[
Z(\mu) = C + \tfrac 14 \sum_k \Re G_k,
\]
where $C = \tfrac 14 \sum_k (L_k + E_k + H_k + F_k)$ is finite. Since $G_k \le 0$, we conclude that $Z(\mu) > -\infty$ is equivalent to $\sum_k G_k > -\infty$, i.e., to $(S^2-2S+1)\alpha \in \ell^4$. This completes the proof of Theorem~\ref{T1.3}.

\bibliographystyle{amsplain}

\providecommand{\bysame}{\leavevmode\hbox to3em{\hrulefill}\thinspace}
\providecommand{\MR}{\relax\ifhmode\unskip\space\fi MR }
\providecommand{\MRhref}[2]{%
  \href{http://www.ams.org/mathscinet-getitem?mr=#1}{#2}
}
\providecommand{\href}[2]{#2}

\end{document}